\documentclass[12pt]{amsart}
\usepackage[cal=rsfso,calscaled=.96]{mathalfa}

\makeatletter
\def\subsection{\@startsection{section}{1}%
  \z@{1.2\linespacing\@plus\linespacing}{.5\linespacing}%
  {\normalfont\scshape\centering}}
\def\subsection{\@startsection{subsection}{2}%
  \z@{0.9\linespacing\@plus.7\linespacing}{-.5em}%
  {\normalfont\bfseries}}
\makeatother

\usepackage{amsmath}
\usepackage{amsfonts}
\usepackage{amsbsy}
\usepackage{amssymb}
\usepackage[normalem]{ulem}
\usepackage{times}
\usepackage{mathrsfs}
\usepackage{exscale}
\usepackage{graphicx}
\usepackage[colorlinks,citecolor=blue,linkcolor=rouge,
            bookmarksopen,
            bookmarksnumbered
           ]{hyperref}
\usepackage{dsfont}
\usepackage{times}

\usepackage{color}
\definecolor{gr}{rgb}   {0.,   0.6,   0.25 }
\definecolor{mg}{rgb}   {0.85,  0.,    0.85}
\definecolor{marin}{rgb}   {0.,   0.,   0.8}
\definecolor{rouge}{rgb}   {0.8,   0.,   0.}
\definecolor{orange}{rgb}   {0.8,   0.4,   0.}

\usepackage{booktabs}

\newtheorem{theorem}{Theorem}[section]
\newtheorem{lemma}[theorem]{Lemma}
\newtheorem{proposition}[theorem]{Proposition}

\theoremstyle{definition}

\theoremstyle{remark}

\numberwithin{equation}{section}

\newcommand{\ee}{\hskip0.15ex}

\newcommand{\dd}[1]{_{\raise-0.3ex\hbox{$\scriptstyle #1$}}}

\newcommand{\on}[1]{\raise-.5ex\hbox{\big|}_{#1}}
\renewcommand{\Re}{\operatorname{\rm Re}}

\newcommand{\pv}{\mathop{\rm p.v.}}

\newcommand {\DNorm}[2]{ \mathchoice
    {\|\ee #1\ee\|\dd{#2}}
    {\| #1 \|_{#2}}
    {\| #1 \|_{#2}}
    {\| #1 \|_{#2}} }

\newcommand {\DNormc}[2]{ \mathchoice
    {\|\ee #1\ee\|\dd{#2}^2}
    {\| #1 \|_{#2}^2}
    {\| #1 \|_{#2}^2}
    {\| #1 \|_{#2}^2} }

\newcommand\C{{\mathbb C}}
\newcommand\R{{\mathbb R}}
\newcommand\N{{\mathbb N}}

\newcommand\Z{{\mathbb Z}}

\renewcommand{\Re}{\operatorname{Re}}

\newcommand\cC{{\mathcal{C}}}

\newcommand\cL{{\mathcal{L}}}

\newcommand {\Id}{\mathbb I}

\newcommand{\dist}{\mathrm{dist}}

\newcommand{\sign}{\mathrm{sign}}

\newcommand{\Span}{\mathrm{span}}

\newcommand{\ov}{\overline}

\usepackage{soul}

\title[A simple discretization of the finite Hilbert transformation]{
Convergence of a simple discretization of \\the finite Hilbert transformation
}
\author{Martin Costabel}
\address{Univ. Rennes, CNRS, IRMAR - UMR 6625, F-35000 Rennes, France}
\email{Martin.Costabel@univ-rennes1.fr}


\date{04/09/2023}

\let\hat=\widehat

\begin{document}

\begin{abstract}
For a singular integral equation on an interval of the real line, we study the behavior of the error of a delta-delta discretization. We show that the convergence is non-uniform, between order $O(h^{2})$ in the interior of the interval and a boundary layer where the consistency error does not tend to zero.
\end{abstract}

\maketitle

%%%%%
\section{Introduction}\label{S:intro}
%%%%%
Let $a,b\in\R$ with $a<b$. On the interval $\Omega=(a,b)$ we consider the singular integral equation, abbreviated as  $(\lambda\Id-A_{\Omega})u=f$,
\begin{equation}
\label{E:SIE}
 \lambda u(x) - \pv\int_{\Omega}\frac{u(y)}{i\pi(x-y)}dy = f(x)\,,\quad x\in\Omega\,.
\end{equation}
We discretize it in the simplest imaginable fashion. We choose $N\in\N$ defining the mesh width $h=\frac1N$ and fix some origin $a^{N}\in\R$, thus  defining the infinite regular grid
$$
   \Sigma^{N}=\{x^{N}_{m}=a^{N}+mh \mid m\in\Z\}
$$
and its finite counterpart $\Sigma^{N}\cap\Omega$ indexed by
$$
 \omega^{N} = \{m\in\Z \mid x^{N}_{m}\in \Omega\}\,,
$$
and consider the system
\begin{equation}
\label{E:DDD}
 \lambda u_{m} - \frac1{i\pi N}\sum_{n\in\omega^{N},m\ne n}\frac{u_{n}}{x^{N}_{m}-x^{N}_{n}}= f(x^{N}_{m}) \,,\quad (m\in\omega^{N})\,.
\end{equation}
This system can easily be programmed and solved with a couple of lines of code, without any knowledge of analysis, and it produces surprisingly good results, see Figure~\ref{F:DDD} for some examples for which the exact solution of the integral equation is known. 
These examples are introduced below in Section~\ref{SS:Expl} and further analyzed numerically in Section~\ref{S:Num}.
\begin{figure}
\centering
\includegraphics[width=0.32\linewidth]{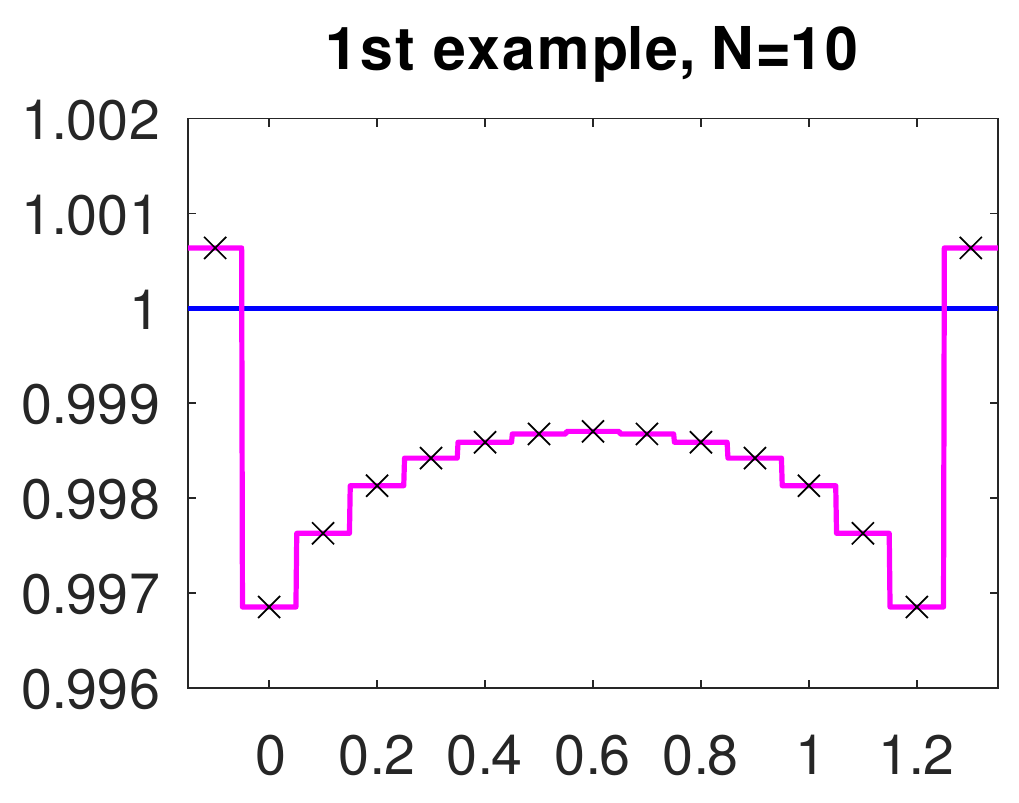}\hfill
\includegraphics[width=0.30\linewidth]{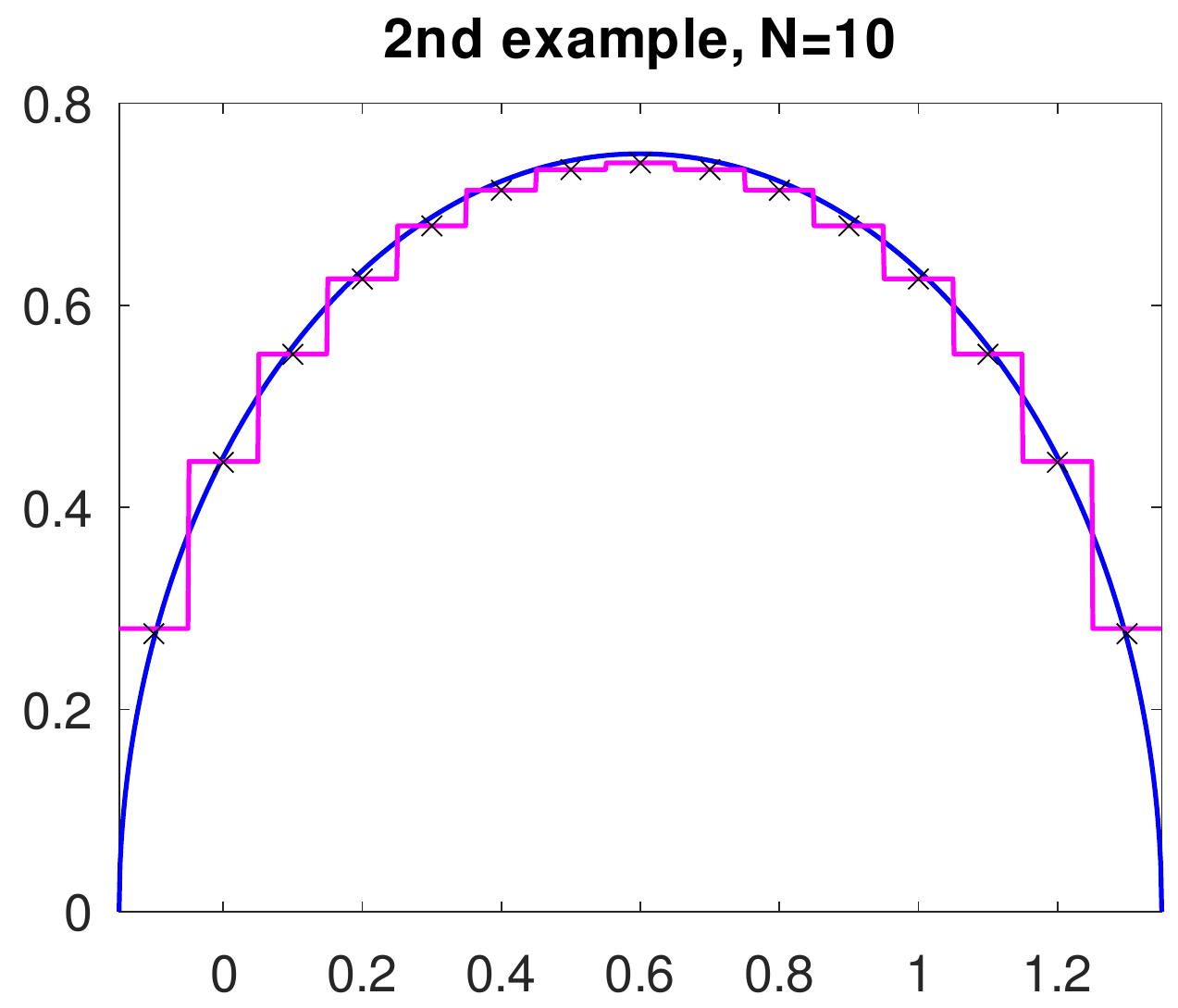}\hfill
\includegraphics[width=0.30\linewidth]{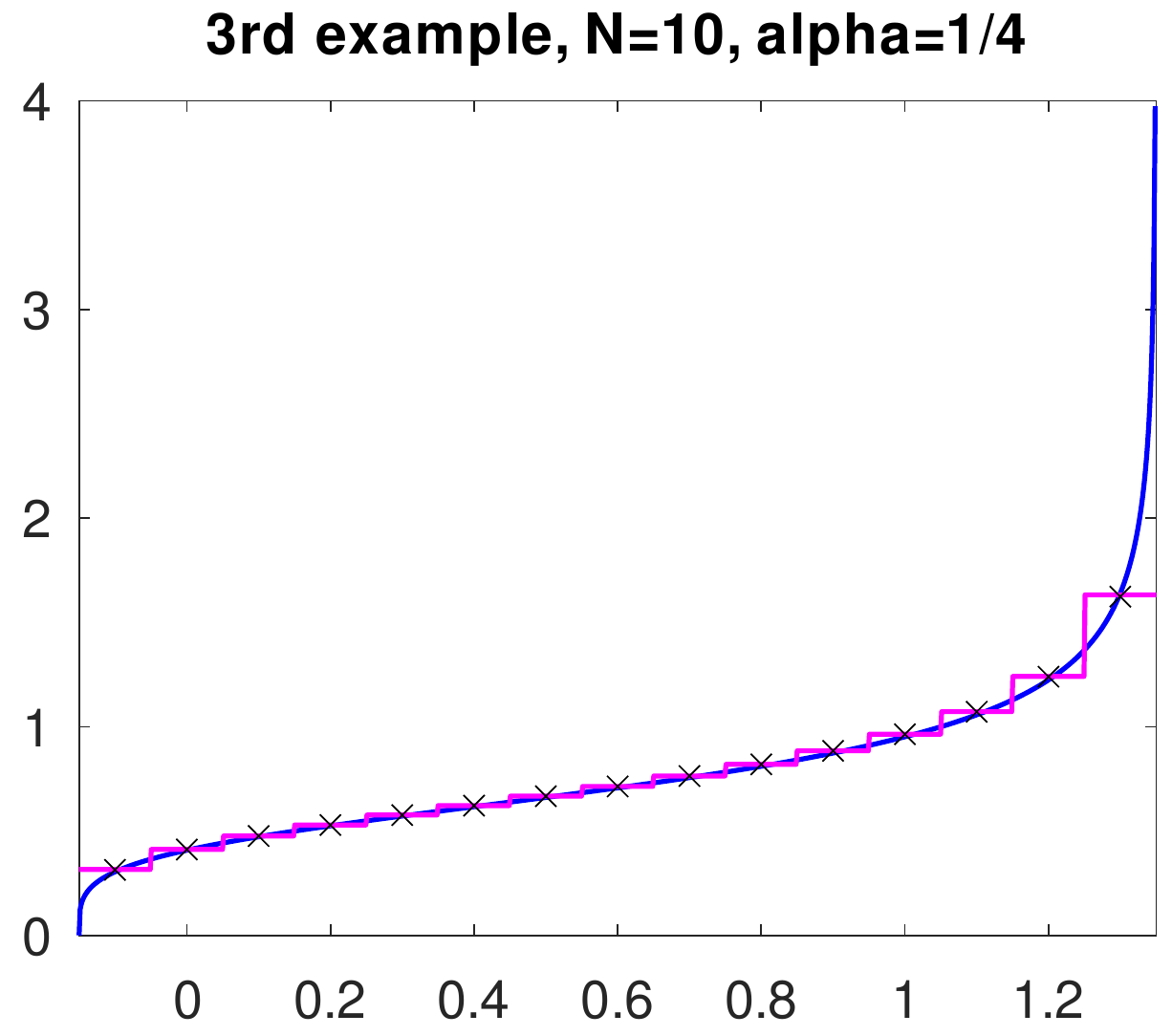}
\caption{3 examples: Exact solutions and their computed p/w constant approximations}
\label{F:DDD}
\end{figure}

This simple delta-delta scheme is similar to Lifanov's method of discrete vortices \cite{BelotserkovskyLifanov1993}, except that we take the same grid for the quadrature points and the evaluation points, and we put zero on the diagonal. 

There is also a similarity to the fully discrete Calder\'on calculus analyzed in \cite{sayas2014-b}, although the differences, namely that we consider an open interval and not a closed curve and that our integral operator is strongly singular, are too important to try a similar analysis.

Note that the approximation scheme \eqref{E:DDD} for the integral equation \eqref{E:SIE} is not
a projection method (or Petrov-Galerkin scheme) in any meaningful sense, although it has the form (except for the diagonal term) of a Galerkin scheme with Dirac deltas as test and trial functions. This has the negative consequence that some tools are not available that one would like to use in order to generalize results proved for the model singular integral equation \eqref{E:SIE} to equations with more general strongly singular kernels, a variable multiplier $\lambda$ instead of a constant, or equations with lower order terms. These tools use the persistence of stability under compact perturbations of the operator, and this is known for projection methods and also for more general discrete approximation schemes in the Stummel-Vainikko sense, see for example \cite{Vainikko1981}. But it seems that our simple scheme does not fit into any of these frameworks.  

The question of a convergence proof for the delta-delta approximation \eqref{E:DDD} came up in the context of our recent research into the error analysis of higher-dimensional simple discretization methods for strongly singular volume integral equations related to the Discrete Dipole Approximation (DDA). The latter has been a standard tool in computational electrodynamics for half a century \cite{PurcellPennypacker1973,YurkinHoekstra2007,Chaumet_M_22}, with a wide range of applications from interstellar dust clouds to nano-particles. Despite the popularity of DDA in computational physics, there is very little known about its mathematical properties, but we have now some results about its stability \cite{CoDaNeStab1,CoDaNeStab2}. In order to complete the convergence proof, one needs estimates for the consistency error, and for gaining insight on the behavior of this, the one-dimensional example \eqref{E:SIE}, \eqref{E:DDD} presented itself as a ``toy problem'', where the error analysis should be more transparent, while still showing some essential peculiarities of the more complicated higher-dimensional situation. It turns out that this is indeed the case, but that the results are interesting by themselves, and we present this in the following.

%%%%%%%%%
\section{Stability}\label{S:Stab}
%%%%%%%%%

Numerical stability of the system \eqref{E:DDD}, which we abbreviate as 
$(\lambda\Id-T^{N})U^{N}=F^{N}$, means that there exists a bound of $U^{N}$ by $F^{N}$ uniform in $N$, that is, a uniform resolvent estimate in some operator norm
$\|(\lambda\Id-T^{N})^{-1}\|\le C_{S}$ for all $N$.

%%%%%%%%%
\subsection{Stability in $\ell^{2}$}\label{SS:Stab}
%%%%%%%%%
Such a stability estimate has been proved in \cite{CoDaNeStab1}, and although most of it is based on quite well known arguments, for the sake of completeness we quote the result and the main points of its proof here.

\begin{proposition}
 \label{P:Stab1D}
The matrix $T^{N}$ of the system \eqref{E:DDD} is selfadjoint with its eigenvalues in $\cC=[-1,1]$. For $\lambda\in\C$ 
 the discretization method \eqref{E:DDD} is stable in the $\ell^{2}$ norm if and only if the integral operator $\lambda\Id-A_{\Omega}$ is boundedly invertible in $L^{2}(\Omega)$, and this is equivalent to 
 $\lambda\not\in\cC$.
 For such $\lambda$, there is an estimate for the operator norms 
 \begin{equation}
\label{E:Res1D}
 \|(\lambda\Id-T^{N})^{-1}\|_{\cL(\ell^{2}(\omega^{N}))} \le \dist(\lambda,\cC)^{-1}
 = \|(\lambda\Id-A_{\Omega})^{-1}\|_{\cL(L^{2}(\Omega))}\,.
\end{equation}
\end{proposition}

The arguments are based on Fourier analysis and on a Lax-Milgram or Galerkin style use of the notion of numerical range. We recall the definition of the numerical range $W(B)$ of a bounded linear operator $B$ in Hilbert space, namely the range of values on the unit sphere of the sesquilinear form associated with $B$
\begin{equation}
\label{E:W(B)}
 W(B) = \{(u,Bu) | \|u\|=1\}\,.
\end{equation}
This is a bounded set contained in the disk with radius $\|B\|$, convex by the Toeplitz-Hausdorff theorem, and its closure contains the spectrum of $B$. 
Writing for any $u$ of norm one and $z=(u,Bu)$ the estimate
$$
  \dist(\lambda,W(B)) \le |\lambda-z| = |(u,\lambda u-Bu)| \le \|(\lambda\Id-B)u\|\,,
$$
one immediately gets the resolvent estimate in the operator norm
\begin{equation}
\label{E:resB}
  \|(\lambda\Id-B)^{-1}\| \le \dist(\lambda,W(B))^{-1}\,.
\end{equation}
Considering that the restriction to a subspace does not increase the numerical range, one gets the same resolvent estimate for operators defined from $B$ by restricting the sesquilinear form to a subspace. This is the Lax-Milgram argument for stability of Galerkin methods, and it applies here both to the integral equation \eqref{E:SIE} and the discrete system \eqref{E:DDD}. 

The integral operator $A_{\Omega}$ in \eqref{E:SIE} and the system matrix $T^{N}$ in \eqref{E:DDD} have this feature in common: They are both projections to bounded sets of translation invariant (i.e.\ convolution) operators, and these can be diagonalized by Fourier analysis. 

The finite Hilbert transformation $A_{\Omega}$ is the restriction to $\Omega$ of the Hilbert transformation $A$ on $\R$, that is the convolution with the kernel $K$ defined as the distribution 
$$
  K(x) = \frac1{i\pi}\pv\frac1{x}
$$
that has the Fourier transform
$$
 \hat K(\xi)=\sign\,\xi \quad(\xi\in\R)\,.
$$
From the Plancherel theorem follows that the convolution operator $A$ acting in the Hilbert space $L^{2}(\R)$ is equivalent to the operator of multiplication with its symbol $\hat K$ in $L^{2}(\R)$. It follows in particular that $A$ is a selfadjoint involution and its spectrum consists of two eigenvalues $\pm1$ of infinite multiplicity. 

This implies $W(A_{\Omega})\subset W(A)=\cC$ and the resolvent estimate
$$
 \|(\lambda\Id-A_{\Omega})^{-1}\|_{\cL(L^{2}(\Omega))}\le \dist(\lambda,\cC)^{-1}\,.
$$

On the discrete side there is a parallel argument using Fourier series instead of the Fourier transform.
The system matrix $T^{N}$ is a finite section of a (bi-)infinite Toeplitz matrix $T$ that, owing to
$$
   \frac1{i\pi N}\frac1{x^{N}_{m}-x^{N}_{n}} = \frac1{i\pi(m-n)}\,,
$$
is independent of $N$,
\begin{equation}
\label{E:Toep}
  T = \big( \frac1{i\pi(m-n)} \big)_{m,n\in\Z}\quad
  \mbox{ with zero on the diagonal.}
\end{equation}
Let $\sigma_{T}$ be its symbol (characteristic function) defined by the Fourier series
$$
 \sigma_{T}(\tau) = \sum_{m\in\Z} \frac1{i\pi m} e^{im\tau} 
       =  \sum_{m=1}^{\infty} \frac{2\sin m\tau}{\pi m}= \sign \,\tau - \frac \tau\pi\quad (\tau\ne0)\,,\quad \sigma_{T}(0)=0\,.
$$
Then by the Plancherel theorem it follows that the discrete convolution operator defined by $T$ acting on $\ell^{2}(\Z)$ is equivalent to the operator of multiplication with $\sigma_{T}$ in $L^{2}(-\pi,\pi)$. Therefore both its spectrum and the closure of its numerical range $\overline{W(T)}$ are equal to $\cC=[-1,1]$. Since $T^{N}$ arises from $T$ by restriction to $\ell^{2}(\omega^{N})$, we get the inclusion $W(T^{N}) \subset \cC$ and the uniform resolvent estimate in \eqref{E:Res1D}.

With these simple arguments, we have proved Proposition~\ref{P:Stab1D}, except for two points:
The assertion that stability implies that $\lambda\not\in\cC$, and the last equality in \eqref{E:Res1D}, where we would only get an inequality. 

For the first point we use the fact that whereas the system \eqref{E:DDD}
\begin{equation}
\label{E:discfin}
(\lambda\Id-T^{N})U^{N}=F^{N}\,,
\end{equation}
when considered as an approximation scheme for the integral equation \eqref{E:SIE}
$(\lambda\Id-A_{\Omega})u=f$, is not a projection method, it is indeed a standard Galerkin scheme when considered as an approximation scheme for the infinite system  
\begin{equation}
\label{E:discinf}
(\lambda\Id-T)U=F\quad \mbox{ in } \ell^{2}(\Z)\,.
\end{equation}
As such, it converges whenever it is stable; in detail:
Denote by $P_{N}$ the orthogonal projection operator in $\ell^{2}(\Z)$ onto the subspace of sequences vanishing outside $\omega^{N}$, identified with $\ell^{2}(\omega^{N})$, so that we can write
$(\lambda\Id-T^{N})P_{N}=P_{N}(\lambda\Id-T)P_{N}$.
If we take as the origin $a^{N}$ of the grid a fixed point inside $\Omega$, then 
$$
  \lim_{N\to\infty}P_{N}U = U \quad \mbox{ in $\ell^{2}(\Z)$ for any $U$}.
$$
 If $U^{N}$ and $U$ satisfy \eqref{E:discfin}, \eqref{E:discinf} with 
 $F^{N}=P_{N}F$, then we have the identity
$$
  U-U^{N} = \big(\Id - (\lambda\Id-T^{N})^{-1}P_{N}(\lambda\Id-T)\big) (U-P_{N}U)\,.
$$
Now let $\lambda\in\C$ be such that the scheme \eqref{E:discfin} is stable,
$$
  \|(\lambda\Id-T^{N})^{-1}\| \le C_{S}\,.
$$
Then $\|U-U^{N}\| \le \big(1+C_{S}\|\lambda\Id-T\|\big)\|U-P_{N}U\|$ (this is the quasi-optimality of the Galerkin scheme, Céa's lemma), hence $U^{N}\to U$, and because of $\|U^{N}\|\le C_{S}\|F\|$, we find in the limit
$$
  \|U\|\le C_{S}\|F\|\,.
$$
This implies that $\lambda\Id-T$ is  invertible with the norm of the inverse bounded by $C_{S}^{-1}$, hence $\lambda\not\in\cC$, in fact 
$\dist(\lambda,\cC)\ge C_{S}^{-1}$\,.

For the  second point, we recall that
the finite Hilbert transformation and its spectral theory is a well-studied classical object. In particular it can be diagonalized by a generalized Fourier transformation \cite{KoppelmanPincus1959} implying that as soon as $\Omega$ is a proper subinterval of $\R$ (even a half line), its action in $L^{2}(\Omega)$ is unitarily equivalent to the operator of multiplication by $\sigma$ in $L^{2}(-1,1)$ with $\sigma(\xi)=\xi$.
Both its spectrum and the closure of its numerical range $\overline{W(A_{\Omega})}$ are therefore equal to 
 $\cC$, which shows the last equality in \eqref{E:Res1D}.

%%%%%%%%%
\subsection{Solution behavior}\label{SS:Sol}
%%%%%%%%%
The simplicity of the statement of Proposition~\ref{P:Stab1D} is somewhat deceptive and hides a more complicated situation: Whereas the solvability of the finite system \eqref{E:DDD} is, of course, independent of the norm we choose in the space of sequences, the latter only being used to describe the stability, this is quite different for the solvability of the integral equation \eqref{E:SIE}. Here, owing to the singular behavior of the solution at the endpoints of $\Omega$, the spectrum of the finite Hilbert transform depends on the function space.
The generalized eigenfunctions
$$
  g_{\xi}(x) = \tfrac1\pi \sqrt{\tfrac{b-a}{2(x-a)(b-x)}}
  e^{\frac{i}{2\pi}\log\frac{\xi+1}{1-\xi}\log\frac{x-a}{b-x}}
$$ 
for which it is shown in \cite{KoppelmanPincus1959} that they diagonalize the operator via the  Hilbert space isomorphism $f\mapsto F:L^{2}(a,b)\to L^{2}(-1,1)$ with the transform pair
$$
 F(\xi) = \tfrac1{\sqrt{1-\xi^{2}}}\int_{a}^{b}\ov{g_{\xi}(x)}f(x)dx,\quad
  f(x) = \int_{-1}^{1}g_{\xi}(x)F(\xi)\tfrac{d\xi}{\sqrt{1-\xi^{2}}}
$$
such that $A_{\Omega}f$ is transformed to $\xi F(\xi)$, are genuine eigenfunctions in $L^{p}(\Omega)$ with $p<2$. More precisely, one considers the arcs of circles
$$
  C_{\alpha_{0}} = \{\lambda\in\C \mid \tfrac{\lambda+1}{\lambda-1} = e^{2\pi i\alpha}, \Re\alpha=\alpha_{0}\}
$$ 
which connect the points $-1$ and $1$ inside the unit circle if 
$\frac14<\alpha_{0}<\frac34$ and outside if $\alpha_{0}\in(0,\frac14)\cup(\frac34,1)$.
 Then the spectrum of $A_{\Omega}$ in $L^{p}(a,b)$ is the domain 
between $C_{1-\frac1p}$ and $C_{\frac1p}$. Its interior consists of eigenvalues. For an eigenvalue $\lambda$, define the exponent $\alpha$ by the above relation
$$
\alpha = \tfrac1{2\pi i}\log\tfrac{\lambda+1}{\lambda-1} 
  \quad\mbox{ or equivalently }\quad
  \lambda = -i \cot\pi\alpha,\qquad
  1-\tfrac1p < \Re\alpha < \tfrac1p\,.
$$
Then the corresponding eigenfunction in $L^{p}(a,b)$ is
$$
  u(x) = (x-a)^{-\alpha} (b-x)^{\alpha-1}.
$$
The spectrum in the dual space $L^{\frac p{p-1}}(\Omega)$ is the same set. 
Thus for $p<2$ the solution of \eqref{E:SIE} exists, in general, but is not unique, and for $p>2$ it is unique, but does not exist for all right hand sides.
Only for $p=2$, the two arcs of circle degenerate to the interval $\cC$, and for $\lambda\in\cC$ the operator $\lambda\Id-A_{\Omega}$ in $L^{2}(\Omega)$ is injective with dense range.

Thus, the delta-delta discretization manages to mimick the solution behavior of the singular integral equation for the specific value $p=2$. 

For $A_{\Omega}$, it is also known how to write the resolvent explicitly as a singular integral operator with weight, involving multiplication by powers of the distance to the points $a$ and $b$, see for example \cite{Soehngen1954} or \cite{Tricomi1957}. For the discrete version $T^{N}$, such explicit formulas for the diagonalization and resolvent do not seem to be available.  
 
 %%%%%%%%%
\subsection{Exact solutions}\label{SS:Expl}
%%%%%%%%%
Examples where both $u$ and $f$ are explicitly known functions can be obtained from function-theoretic arguments.
For $u\in L^{1}(a,b)$ define the Cauchy integral for $z\in\C\setminus[a,b]$
$$
  w(z) = \frac1{2\pi i}\int_{a}^{b}\frac{u(y)}{y-z}dy\,.
$$
Then $w$ is holomorphic in $\C\setminus[a,b]$, behaves as $O(|z|^{-1})$ as $|z|\to\infty$, and satisfies on $\Omega=(a,b)$  the jump relations
\begin{equation}
\label{E:jump}
 w_{+} - w_{-} = u\,;\qquad w_{+}+w_{-} = A_{\Omega}u,
 \quad \mbox{ where }\quad
 w_{\pm}(x) = \lim_{\varepsilon\to0+} w(x\pm i\varepsilon)\,. 
\end{equation}
Therefore the integral equation \eqref{E:SIE} is equivalent to the jump condition
$$
 (\lambda-1) w_{+} - (\lambda+1) w_{-} =f\,.
$$
Reciprocally, any function $w$ that is holomorphic outside $[a,b]$, vanishes at infinity and has square integrable upper and lower traces on the branch cut $[a,b]$ provides a solution $u$ to the integral equation \eqref{E:SIE} with right hand side $f$, if $u$ and $f$ are defined from the traces $w_{\pm}$ via formulas \eqref{E:jump}.

For example, the function $w_{0}(z)=\log\frac{z-a}{z-b}$ is (or can be chosen to be) holomorphic in $\C\setminus[a,b]$, is $O(|z|^{-1})$ at infinity, and has the upper and lower traces
$$ 
 w_{0,+}(x)= \log\tfrac{x-a}{b-x} -i\pi\,,\quad w_{0,-}(x)= \log\tfrac{x-a}{b-x} +i\pi
$$
Since these traces belong to $L^{2}(a,b)$, the jump
$u_{0}(x)=-2\pi i$ solves the integral equation \eqref{E:SIE} if we choose the right hand side 
$f_{0}(x) = -2\pi i \lambda -2 \log\frac{x-a}{b-x}$.

Another example is
$w_{1}(x)= \sqrt{(z-a)(z-b)}-z+\tfrac{a+b}2$ that has the right branch cut and behavior at infinity. Traces on $(a,b)$ are
$$ 
 w_{1,+}(x)= i\sqrt{(x-a)(b-x)}-x+\tfrac{a+b}2\,,\quad 
 w_{1,-}(x)= -i\sqrt{(x-a)(b-x)}-x+\tfrac{a+b}2
$$
The jump $u_{1}(x)=2i\sqrt{(x-a)(b-x)}$ therefore solves the integral equation \eqref{E:SIE} if we choose the right hand side
$f_{1}(x)=2i\lambda\sqrt{(x-a)(b-x)} + 2x -a-b$.

Other examples can be obtained by applying an entire analytic function $F$ to the preceding example functions $w$, provided $F(0)=0$. 
For a third example, we take the function $F(w)=e^{\alpha w}-1$ and apply it to the first example above, that is, we choose an exponent $\alpha\in\C$ and define
$w_{2}(z)=e^{\alpha\log\frac{z-a}{z-b}}-1=\big(\frac{z-a}{z-b}\big)^{\alpha}-1$.\\
The traces  on the branch cut $w_{2,\pm}(x)=(x-a)^{\alpha}(b-x)^{-\alpha}e^{\mp i\pi\alpha}-1$ belong to $L^{2}(a,b)$ if $\alpha\in(-\frac12,\frac12)$. 
The corresponding exact solution $u_{2}$ and right hand side $f_{2}$ are given by 
$$
  u_{2}(x) = -2i\sin\pi\alpha \Big(\frac{x-a}{b-x}\Big)^{\alpha}\quad\mbox{ with }
  A_{\Omega}u_{2}(x) = 2\cos\pi\alpha \Big(\frac{x-a}{b-x}\Big)^{\alpha}-2 \;\mbox{ and }
  f_{2} = \lambda u_{2} - A_{\Omega}u_{2}\,.
$$
Note that if $\lambda=i\cot\pi\alpha$, then $f_{2}$ is a constant function.
 
%%%%%%%%%
\section{Error estimates}\label{S:Err}
%%%%%%%%%

%%%%%%%%%
\subsection{Approximate solutions}\label{SS:Approx}
%%%%%%%%%
 For sake of simplicity we assume from now on that the length of $\Omega$ is a multiple of $h$ and that the origin $a^{N}$ is chosen such that the boundary points $a$ and $b$ are midpoints of the mesh.
\begin{equation}
\label{E:abmesh}
 b-a= Mh\,,\quad
 M\in\N\,,\quad
 \omega^{N} = \{1,\dots,M\}\,,\quad
 x_{1}^{N} = a+\tfrac h2\,,\quad
 x_{M}^{N} = b-\tfrac h2\,.
\end{equation}
When no confusion is possible, we omit superscripts $N$ and write the discrete system as
\begin{equation} 
\label{E:DDA1D}
 \lambda u_{m}
 -\frac1{i\pi N}\sum_{\begin{subarray}{l}n=1\\n\ne m\end{subarray}}^{M}\frac{u_{n}}{x_{n}-x_{m}} 
 = f_{m}\,,
 \quad m\in\omega^{N}\,.
\end{equation} 
In the spirit of
Lifanov's method of discrete vortices \cite{BelotserkovskyLifanov1993}, 
the delta-delta scheme can be considered as a Nystr\"om-type quadrature method, where the integral in \eqref{E:SIE} is replaced by its approximation by the midpoint rule.
The approximate solution $u^{(N)}$ satisfies the modified equation\begin{equation}
\label{E:Ny}
 \lambda u^{(N)}(x) - \frac1{i\pi}\sum_{n\in\omega^{N}}\frac{1}{N}\frac{u^{(N)}(x_{n})}{x_{n}-x} = f(x) \,,
\end{equation}
and its evaluation in the gridpoints $x=x_{m}$ gives the system \eqref{E:DDA1D} for the nodal values $u_{m}=u^{(N)}(x_{m})$.

Alternatively, one can also consider the delta-delta scheme as a one-point quadrature approximation of a collocation method with piecewise constant trial functions. This is the popular point of view for DDA, see for example the error analysis in \cite{Yurkin:06i}.

Let $\chi_{j}$ be the characteristic function of the interval 
$I_{j}=[x_{j}-\frac h2,x_{j}+\frac h2]$.
We consider the space of piecewise constant functions
$S^{N,0}(\omega^{N})=\Span\{\chi_{j}\mid j\in\omega^{N}\}$.

Then the collocation scheme is:
Find $u^{N}\in S^{N,0}(\omega^{N})$ such that 
\begin{equation}
\label{E:Col}
 \lambda u^{N}(x_{m}) - \frac1{i\pi}\int_{\Omega}\frac{u^{N}(y)}{y-x_{m}}dy = f(x_{m})\,,\quad m\in\omega^{N} \,.
\end{equation}
From this we get our system \eqref{E:DDA1D} by 
taking the $u_{n}$ as the coefficients of $u^{N}$ in the basis of the $\chi_{n}$,
$$
  u^{N} = \sum_{n\in\omega^{N}}u_{n}\chi_{n},
$$
and for the off-diagonal matrix elements $A_{\Omega}\chi_{n}(x_{m})$  approximating the integrals by a one-point quadrature rule:
$$
 \int \frac{\chi_{k}(y)}{y-x_{m}}dy \;\sim\; 
 \begin{cases}0 & \text{ if $m=n$}\\
    \dfrac1{x_{n}-x_{m}}\,\dfrac1N & \text{ if $m\not=n$}
 \end{cases}
$$
After this replacement of the integrals, the system \eqref{E:Col} becomes \eqref{E:DDA1D}. 
For the diagonal elements, note that the Cauchy principal value integral
$\int \frac{\chi_{n}(y)}{y-x_{n}}dy$ vanishes for symmetry reasons. 

Notice that even if the coefficients $(u_{n})$ in \eqref{E:Ny}  and \eqref{E:Col} are the same, the approximate solutions $u^{N}$ and $u^{(N)}$ are not the same: $u^{N}$ in \eqref{E:Col} is a piecewise constant function, whereas $u^{(N)}$, which for $\lambda\ne0$ is defined by \eqref{E:Ny} once the $u_{k}$ are known, is a sum of the possibly smooth function $f$ and a rational function with simple poles in the grid points. 

%%%%%%%%%
\subsection{Consistency error}\label{SS:ConsErr} 
%%%%%%%%%
We define 
the \emph{consistency error} as the vector $c^{N}[u]$ with components
\begin{equation}
\label{E:cjN}
  c_{m}^{N}[u] =  \frac1{i\pi}\Big(\int_{a}^{b} \frac{u(y)}{y-x_{m}}dy 
  - h\sum_{n\in\omega^{N}\setminus\{m\}}\frac{u(x_{n})}{x_{n}-x_{m}}\Big),
  \quad m\in\omega^{N}\,.
\end{equation}
It can be written as $c^{N}[u]=(R_{N}A_{\Omega}-T^{N}R_{N})u$, where $R_{N}$ is the restriction operator that maps a function $u$ to its nodal values
$U=(u(x_{m}))_{m\in\omega^{N}}$. 

Thus it is the quadrature error for the midpoint rectangle rule applied to the Cauchy singular integral. 
This quadrature error has been estimated under various assumptions on the function $u$ in \cite[Section~1.3]{BelotserkovskyLifanov1993}, from which we borrow some ideas. We study here only the case where $u$ is H\"older continuous up to the boundary, $u\in C^{\alpha}([a,b])$ for some $\alpha\in(0,1]$.

In the applications of the singular integral equation \eqref{E:SIE} studied in \cite{BelotserkovskyLifanov1993}, mainly from fluid dynamics, more general functions $u$ with singularities at the boundary points $a$ and $b$ need to be considered.
But since we want to emphasize the analogy with the volume integral equation of the DDA method, and solutions of the latter tend to be smooth up to (smooth points of) the boundary, the simple situation of $u\in C^{\alpha}([a,b])$ will be sufficient for our purpose.

An advantage of this simplification is that we can get convergence results in an easy way from combining the consistency estimates of this section with the stability estimates of the preceding section.
Convergence results for the case of solutions with singularities seem to require a much more complicated analysis, see \cite[Chapter~7]{BelotserkovskyLifanov1993}.

We consider first the case where $u$ is constant.
 
\begin{lemma}
\label{L:const}
Let
$$
 s_{j}= \int_{a}^{b} \frac{dy}{y-x_{j}}
  - \sum_{k\in\omega^{N}\setminus\{j\}}\frac{h}{x_{k}-x_{j}}\,.
$$
Then for all $j\in\omega^{N}$
\begin{equation}
\label{E:const}
 |s_{j}| \le
 \frac{h^{2}}8 \big|\frac1{(x_{j}-a)^{2}}- \frac1{(b-x_{j})^{2}} \big|
 = |\tfrac{a+b}2 -x_{j}| (b-a) \big(\frac{h}{2(x_{j}-a)(b-x_{j})}\big)^{2}\,.
\end{equation}
\end{lemma}
\begin{proof}
 We split
$$
  s_{j} = \sum_{k\in\omega^{N}\setminus\{j\}}s_{jk} \quad\mbox{ with }\;
  s_{jk} = \int_{I_{k}} \frac{dy}{y-x_{j}}
  - \frac{h}{x_{k}-x_{j}}\,.
$$
First we observe that there are cancellations: $s_{jk}+s_{jk'}=0$ for 
$k'+k=2j$, because of the antisymmetry of the kernel. Therefore
$$
s_{j} = \sum_{k\in A_{j}}s_{jk}\quad\mbox{ with }\;
 A_{j} = 
 \begin{cases} \{k \mid 2x_{j}-a +\frac h2 < x_{k} <b\} & \mbox{ if }x_{j} \le \tfrac{a+b}2\,,
    \\[0.5ex]
               \{k \mid a < x_{k} < 2x_{j}-b-\frac h2\} & \mbox{ if }x_{j} \ge \tfrac{a+b}2\,.
 \end{cases}
$$
Then we apply the classical error representation for the midpoint rule 
$$
  \int_{-\frac h2}^{\frac h2} f(t)dt - hf(0) = 
   \int_{-\frac h2}^{\frac h2} \tfrac12 (\tfrac h2-|t|)^{2} f''(t) dt\,,
$$
valid for any $C^{2}$ function,
to the function
$f(t)=\frac1{t+x_{k}-x_{j}}$. 
Consider the case $x_{j} \le \tfrac{a+b}2$. Then for the relevant indices $k$, we have $x_{k}>x_{j}$ and therefore $f''(t)=2(t+x_{k}-x_{j})^{-3}>0$. Hence
$$
 0<s_{jk}<\frac{h^{2}}8\int_{I_{k}}\frac{2dy}{(y-x_{j})^{3}}
 \quad\Longrightarrow\;
 0<s_{j}< \frac{h^{2}}8\int_{2x_{j}-a}^{b}\frac{2dy}{(y-x_{j})^{3}}
 = \frac{h^{2}}8\big( -\frac1{(b-x_{j})^{2}} + \frac1{(x_{j}-a)^{2}}\big)\,.
$$
Likewise, for $x_{j} \ge \tfrac{a+b}2$, the second derivative is always negative, and we obtain
$$
 0>s_{j}> \frac{h^{2}}8\int_{a}^{2x_{j}-b}\frac{2dy}{(y-x_{j})^{3}}
 = \frac{h^{2}}8\big( -\frac1{(b-x_{j})^{2}} + \frac1{(x_{j}-a)^{2}}\big)\,.
$$
\end{proof}
The lemma means that for $x_{j}$ in any compact subinterval of $(a,b)$, the $s_{j}$ converge to zero as $O(h^{2})$. But when $x_{j}$ tends to $a$ or $b$ as $h\to0$, then $s_{j}$ will have non-zero limits. For example for $x_{1}=a+\frac h2$, one finds $s_{1}\to\log2-\gamma$, where $\gamma=0.5772..$ is Euler's constant. For this case, the estimate \eqref{E:const} gives $s_{j}\le\frac12$ in the limit.

Another way to express this behavior is that for a 
given error threshold $\epsilon$, the points $x_{j}$ where $|s_{j}|$ exceeds $\epsilon$ are confined to a boundary layer of thickness $h/\sqrt{2\epsilon}$.

%In \cite[Lemma~1.3.1]{BelotserkovskyLifanov1993}, an estimate
%$$
% |s_{j}| \le C\,\frac{h}{(x_{j}-a)(b-x_{j})}
%$$
%is stated (without proof). This is an easy consequence of \eqref{E:const}.
%

\begin{proposition}
\label{P:cons}
Let $0<\alpha\le1$. 
For $u\in C^{\alpha}([a,b])$ let $c_{j}^{N}[u]$ be defined by \eqref{E:cjN}. 
Then
\begin{equation}
\label{E:consCalpha}
 |c_{j}^{N}[u]| \le
 C\,\Big(
 h^{\alpha}|\log h| \DNorm{u}{C^{\alpha}([a,b])} 
  + |u(x_{j})| \frac{h^{2}}{(x_{j}-a)^{2}(b-x_{j})^{2}}
 \Big)\,.
\end{equation}
Here the constant $C$ depends on $a,b$ but not on $h=\frac1N$ nor on $u$.\\
If, in addition, $u(a)=u(b)=0$ holds, then the estimate simplifies to
\begin{equation}
\label{E:consC0alpha}
 |c_{j}^{N}[u]| \le
 C\,
 h^{\alpha}|\log h| \DNorm{u}{C^{\alpha}([a,b])} 
 \,.
\end{equation}
\end{proposition}
\begin{proof}
Split
\begin{equation*}
\begin{split}
 i\pi c_{j}^{N}[u] &= c_{jj} + \sum_{k\ne j}c_{jk} + u(x_{j})s_{j}\quad\mbox{ with }\\
 c_{jj} &= \int_{I_{j}}\frac{u(y)}{y-x_{j}}dy,\;
 c_{jk} = \int_{I_{k}}\big(\frac{u(y)-u(x_{j})}{y-x_{j}} - \frac{u(x_{k})-u(x_{j})}{x_{k}-x_{j}}
 \big)dy
 \end{split}
\end{equation*}
and $s_{j}$ as defined and estimated in Lemma~\ref{L:const}.\\
For $c_{jj}$ we find
$$
|c_{jj}|=\big|\int_{I_{j}}\frac{u(y)-u(x_{j})}{y-x_{j}}dy\big| \le
\int_{I_{j}}|y-x_{j}|^{\alpha-1}dy\, \DNorm{u}{C^{\alpha}([a,b])} \le
h^{\alpha}\, \DNorm{u}{C^{\alpha}([a,b])} \,.
$$
For the term $u(x_{j})s_{j}$ we use the estimate from Lemma~\ref{L:const}. 
It remains to estimate $\sum_{k\ne j}c_{jk}$.
Write
$$
c_{jk}= c_{jk}^{0} + c_{jk}^{1} \quad\mbox{ with }
c_{jk}^{0} =  \int_{I_{k}}\frac{u(y)-u(x_{k})}{y-x_{j}}dy\,,\;
c_{jk}^{1} =  \int_{I_{k}}\big(u(x_{k})-u(x_{j})\big)
  \big(\tfrac{1}{y-x_{j}}-\tfrac{1}{x_{k}-x_{j}}
 \big)dy.
$$
Then
$$
\Big|\sum_{k\ne j}c_{jk}^{0}\Big| \le
 \big(\frac h2\big)^{\alpha} \DNorm{u}{C^{\alpha}([a,b])}
 \Big(
 \int_{a}^{x_{j}-\frac h2} \frac{dy}{x_{j}-y} + \int_{x_{j}+\frac h2}^{b} \frac{dy}{y-x_{j}}
 \Big) \le
 C\,h^{\alpha}|\log h|\,\DNorm{u}{C^{\alpha}([a,b])}\,.
$$
Finally, $c_{jk}^{1} =  \int_{I_{k}}\big(u(x_{k})-u(x_{j})\big)
  \tfrac{x_{k}-y}{(y-x_{j})(x_{k}-x_{j})}
 dy$
 can be estimated as
$$
 |c_{jk}^{1}| \le \DNorm{u}{C^{\alpha}} \int_{I_{k}}
  \frac{|x_{k}-y|}{|y-x_{j}|x_{k}-x_{j}|^{1-\alpha}} dy
 \le C\, \DNorm{u}{C^{\alpha}}\,h^{2}\, |x_{k}-x_{j}|^{\alpha-2}
  = C\, \DNorm{u}{C^{\alpha}}\,h^{\alpha}\, |k-j|^{\alpha-2}
  .
$$
For $\alpha<1$, the infinite series $\sum_{k\in\Z\setminus\{j\}}|k-j|^{\alpha-2}$ converges, so that
$$
 \Big|\sum_{k\ne j}c_{jk}^{1}\Big| \le C\, h^{\alpha}\, \DNorm{u}{C^{\alpha}}\,.
$$
For $\alpha=1$, we would pick up an extra factor of $\log N=|\log h|$. \\
The estimate \eqref{E:consCalpha} is proved.\\
Finally, if $u(a)=u(b)=0$, then $|u(x_{j)}|\le C\,\DNorm{u}{C^{\alpha}}(x_{j}-a)^{\alpha}(b-x_{j})^{\alpha}$, hence
$$
 |u(x_{j})| \frac{h^{2}}{(x_{j}-a)^{2}(b-x_{j})^{2}}
 \le C\, h^{\alpha} \DNorm{u}{C^{\alpha}} 
   \big( \frac{h}{(x_{j}-a)(b-x_{j})}
   \big)^{2-\alpha}
   \,.
$$
The last factor is uniformly bounded, and this completes the proof of \eqref{E:consC0alpha}
\end{proof}

%%%%%%%%%
\subsection{Discrete error}\label{SS:DiscErr}\ 
%%%%%%%%%
If $u$ is the solution of the integral equation \eqref{E:SIE} with right hand side $f$ and 
$U^{N}=(u_{n})_{n\in\omega^{N}}$ is solution of the discrete system \eqref{E:DDD} with $f_{m}=f(x_{m})$, then we define the \emph{discrete error} $E^{N}$ as vector with the components $E_{m}= u(x_{m}) - u_{m}$, $m\in \omega^{N}$. It is easy to see that $E^{N}$ satisfies the discrete system with the consistency error as right hand side,
\begin{equation}
\label{E:DErrCons}
 (\lambda\Id-T^{N})E^{N} = c^{N}[u]\,.
\end{equation}
By a direct application of our stability and consistency estimates we obtain an estimate for the discrete error.
We choose the situation where the simplified consistency estimate \eqref{E:consC0alpha} holds.
\begin{theorem}    
\label{T:conv0}              
Assume $\lambda\in\C\setminus[-1,1]$. Let $u$ be solution of the integral equation \eqref{E:SIE} with right hand side $f$ and 
$(u_{n})_{n\in\omega^{N}}$ be solution of the discrete system \eqref{E:DDD} with $f_{m}=f(x_{m})$. If we assume that $u\in C^{\alpha}([a,b])$ with $\frac12<\alpha\le1$ and $u(a)=u(b)=0$, then there is a constant $C$ independent of $N$ such that the discrete error	
satisfies
\begin{equation}
\label{E:DErrCov}
  \DNorm{E^{N}}{\ell^{2}(\omega^{N})} \le C\,N^{\frac12-\alpha}\log N\,.
\end{equation}
\end{theorem}
\begin{proof}
From the error equation \eqref{E:DErrCons} and our stability estimate \eqref{E:Res1D} we obtain
$$
 \DNorm{E^{N}}{\ell^{2}(\omega^{N})} \le C\,\DNorm{c^{N}[u]}{\ell^{2}(\omega^{N})}\,.
$$
Now our consistency estimate \eqref{E:consC0alpha} is an $\ell^{\infty}(\omega^{N})$ estimate, so we lose a factor of $\sqrt N$:
$$
  \DNorm{c^{N}}{\ell^{2}(\omega^{N})} \le 
  |\omega^{N}|^{\frac12} \DNorm{c^{N}}{\ell^{\infty}(\omega^{N})} \le
    C\,N^{\frac12} \DNorm{c^{N}}{\ell^{\infty}(\omega^{N})}\le
    C\,N^{\frac12-\alpha}\log N \DNorm{u}{C^{\alpha}([a,b])}\,.
$$
\end{proof}

%%%%%%%%%
\subsection{Convergence in $L^{2}$}\label{SS:ConvL2}\ 
%%%%%%%%%

When the solution $u$ does not vanish on the boundary, we only have the consistency estimate \eqref{E:consCalpha}, and the consistency error $c^{N}[u]$ will not converge to zero in $\ell^{\infty}$ and even less in $\ell^{2}$, hence the discrete error will not converge to zero in $\ell^{2}$, either. Instead, we can study the error of the piecewise constant approximation
$$
 e^{N} = u - u^{N} \quad \mbox{ with }\quad u^{N}=\sum_{k\in\omega^{N}}u_{k}\chi_{k}\,.
$$
Its $L^{2}(\Omega)$ norm can be bounded by
$$\begin{aligned}
 \DNormc{e^{N}}{L^{2}(a,b)}&= 
  \sum_{k\in\omega^{N}} \int_{I_{k}}|u(y)-u_{k}|^{2}dy\\ &\le
  2\sum_{k\in\omega^{N}}
  \Big(\int_{I_{k}}|u(y)-u(x_{k})|^{2}dy + \int_{I_{k}}|u(x_{k})-u_{k}|^{2}dy
  \Big)\\ &\le
  2\sum_{k\in\omega^{N}}
  \Big(h^{2\alpha+1} \DNormc{u}{C^{\alpha}} + h\,|E_{k}|^{2}
  \Big)\\ &\le
  C\, \big( h^{2\alpha}\DNormc{u}{C^{\alpha}([a,b])} 
    + h\,\DNormc{E^{N}}{\ell^{2}(\omega^{n})} \big)\,.
\end{aligned}
$$
Thus we get convergence as soon as we can show that $N^{-\frac12}\DNorm{E^{N}}{\ell^{2}}$ tends to zero. By our $\ell^{2}$ stability estimates, this is equivalent to the fact that
$N^{-\frac12}\DNorm{c^{N}}{\ell^{2}}$ tends to zero. For this, as we have seen before, it would be \emph{sufficient} that the consistency error $\DNorm{c^{N}}{\ell^{\infty}}$ tends to zero. 

But this is not \emph{necessary}: In fact, assume that the consistency estimate \eqref{E:consCalpha} is satisfied. It implies
$$
 |c_{j}^{N}| \le C\,\DNorm{u}{C^{\alpha}([a,b])}\Big( h^{\alpha}|\log h| +
  \max\big\{\frac{h^{2}}{(x_{j}-a)^{2}},\,\frac{h^{2}}{(b-x_{j})^{2}}\big\} \Big)
$$
Now $x_{j}-a$ takes the values $\frac h2$, $\frac{3h}2$, $\frac{5h}2$,..., and therefore the sum
$$
\sum_{j\in\omega^{N}}\big(\frac{h^{2}}{(x_{j}-a)^{2}}\big)^{2}
$$
is bounded independently of $N$ by a constant 
$$
 C'= 16\sum_{j=0}^{\infty}(2j+1)^{-4} < \infty\,.
$$
This implies finally
$$
 h\DNormc{c^{N}}{\ell^{2}(\omega^{N})} \le C\,\DNormc{u}{C^{\alpha}([a,b])}
 \Big(h^{2\alpha}|\log h|^{2} + h\,C'\Big)\,.
$$
We have proved the following error estimate.
\begin{theorem}    
\label{T:conv}              
Assume $\lambda\in\C\setminus[-1,1]$. Let $u$ be solution of the integral equation \eqref{E:SIE} with right hand side $f$, let
$U^{N}=(u_{k})_{k\in\omega^{N}}$ be solution of the discrete system \eqref{E:DDA1D} with $f_{j}=f(x_{j})$, and let
$u^{N}=\sum_{k\in\omega^{N}}u_{k}\chi_{k}$ be the corresponding piecewise constant function.
 If we assume that $u\in C^{\alpha}([a,b])$, then there is a constant $C$ independent of $N$ such that the error	$e^{N}=u-u^{N}$ satisfies
$$
  \DNorm{e^{N}}{L^{2}(a,b)} \le C\,\big(N^{-\alpha}\log N + N^{-\frac12}\big)\,.
$$
\end{theorem}

The last term $O(\sqrt h)$ in this error estimate that prevents the estimate to improve with regularity above $C^{\frac12}$ is due to the boundary layer and the choice of the $L^{2}$ norm for measuring the error. One could choose $L^{p}$ for $1\le p\le\infty$ instead. In the same way as above, one obtains
$$
  \DNorm{e^{N}}{L^{p}(a,b)} \le  C \big( h^{\alpha}\DNormc{u}{C^{\alpha}([a,b])} 
    + h^{\frac1p}\,\DNorm{E^{N}}{\ell^{p}(\omega^{n})} \big)\,.
$$
This would be best for $p=1$, but we are tied to $p=2$ because we depend on our stability estimate.

%%%%%%%%%
\section{Numerical experiments}\label{S:Num}
%%%%%%%%%
In this section, we fix arbitrarily the interval $[a,b]=[-0.15,1.35]$ and choose $N$ as an odd multiple of $10$, so that $[a,b]$ is subdivided in $3N/2$ subintervals of length $h=\frac1N$. 
We use the $3$ examples discussed in section~\ref{SS:Expl}, where we divide the first example by $-2\pi i$, so that $u_{0}(x)=1$. In the second example, we normalize by dividing by $2i$, so that $u_{1}(x)=\sqrt{(x-a)(b-x)}$. 
For the third example, we choose $\alpha=0.25$, so that 
$u_{2}(x)=\Big(\frac{x-a}{b-x}\Big)^{\frac14}/\sqrt2$.
In Figure~\ref{F:DDD}, we saw the $3$ exact solutions with their computed piecewise constant approximation, computed for $N=10$ with $\lambda=2$.

%%%%%%%%%
\subsection{Behavior of consistency error and discrete error}\label{SS:c&E}
%%%%%%%%%
Knowing the exact solution, we can define the consistency error $c^{N}$ as in \eqref{E:cjN} and the discrete error $E^{N}$ as in Section~\ref{SS:DiscErr}. In Figure~\ref{F:Ex1-3,ufcE}, we plot for our 3 examples the absolute values of $c^{N}$ and $E^{N}$ as functions of $x\in[a,b]$ for $N=10,50,250,1250$. In the logarithmic scale, one can see the convergence in the interior of the interval with a power of $N$, whereas near the boundary points there is either much slower convergence for example 2, where $u_{1}(a)=u_{1}(b)=0$, or no convergence at all for example 1 where the solution $u_{0}$ does not vanish at the boundary points, thus illustrating the analysis shown in Section~\ref{S:Err}. Note that we have not presented an analysis for the case of a singular solution such as the one in the third example, but the numerical results show a similar behavior as for the other examples, although the consistency and discrete errors, like the solution itself, tend to infinity near $b$.
\begin{figure}[h]
\centering
\includegraphics[width=0.32\linewidth]{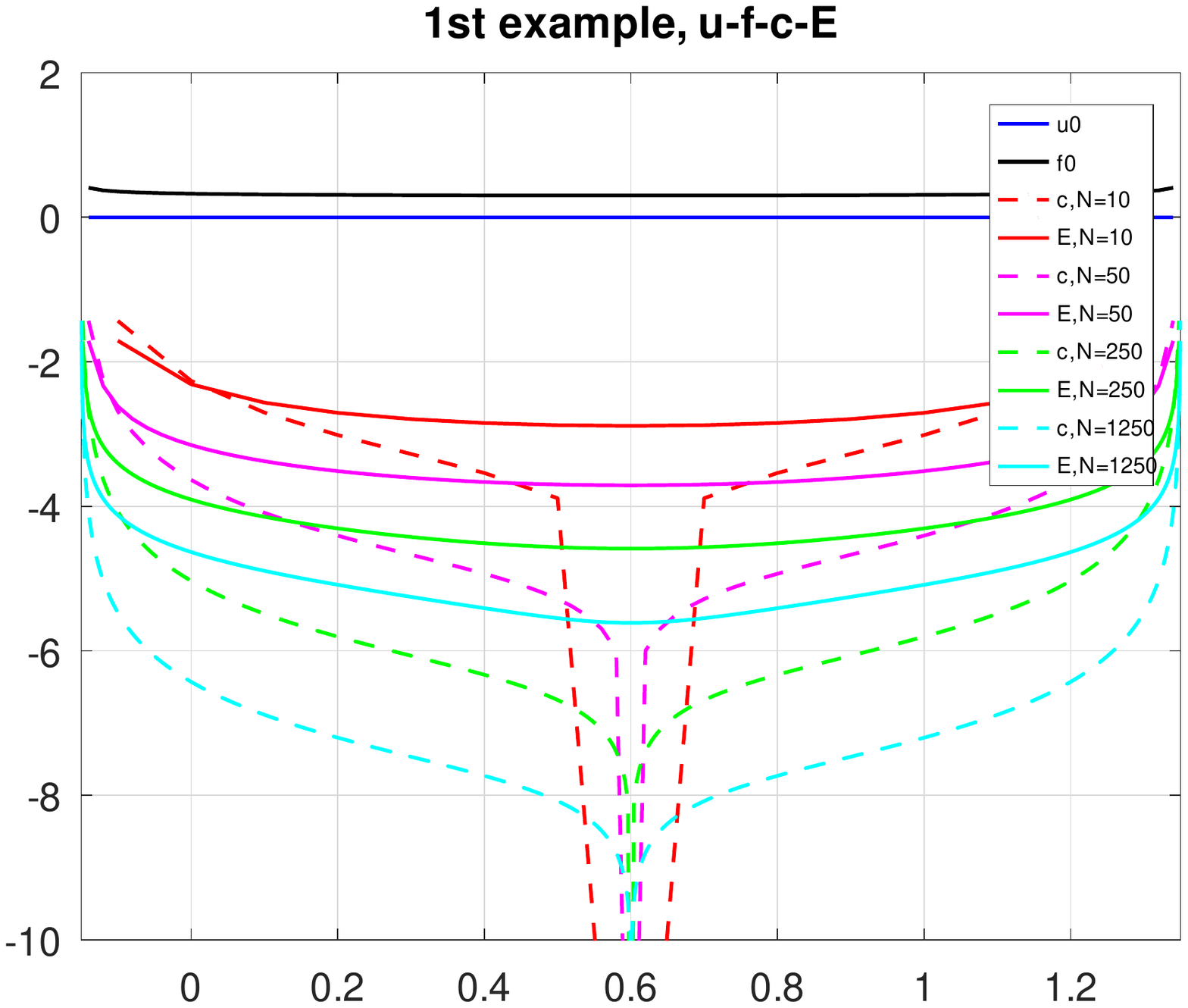}
\includegraphics[width=0.32\linewidth]{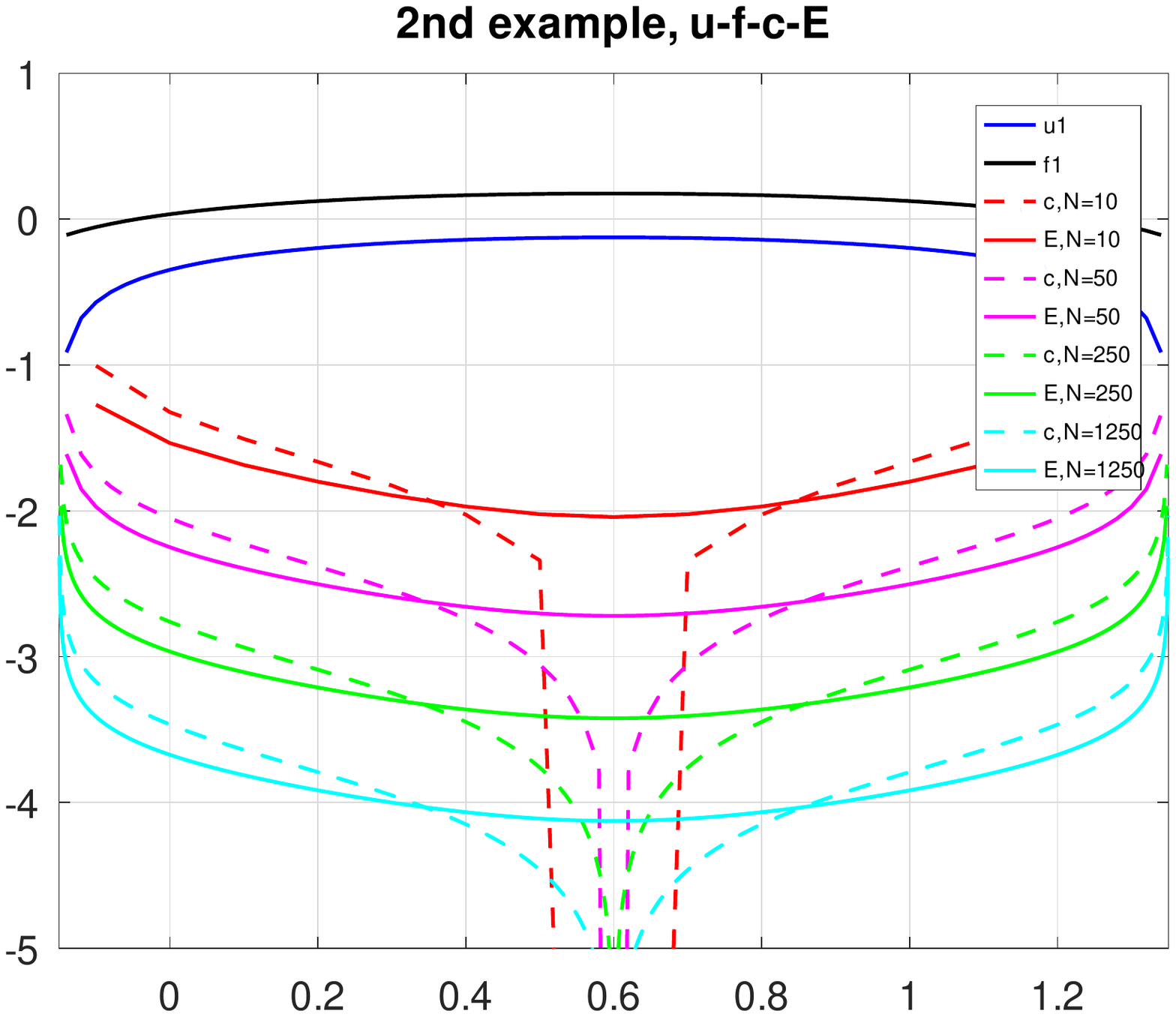}
\includegraphics[width=0.32\linewidth]{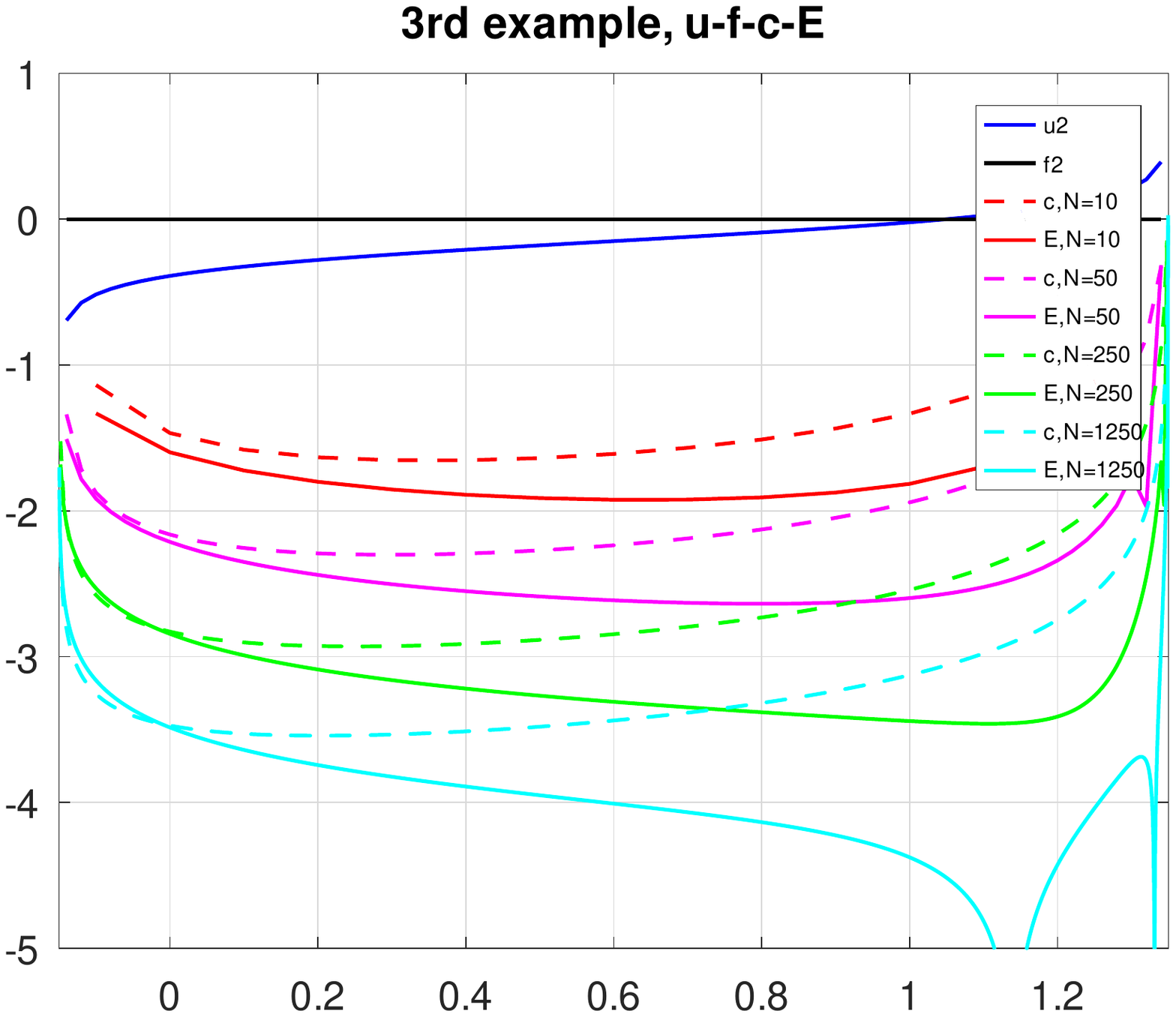}
\caption{Exact solution \textbf{u}, rhs \textbf{f}, consistency error \textbf{c} and discrete error \textbf{E}}
\label{F:Ex1-3,ufcE}
\end{figure}

%%%%%%%%%
\subsection{Convergence rates}
%%%%%%%%%
In Figure \ref{F:Ex1-3,errnorms} we plot several error norms as functions of $N=10\times3^{j}$, $j=0,...,6$ in a loglog scale:
The $\ell^{2}(\omega^{N})$ norms of the consistency error $c^{N}$ and the discrete error $E^{N}$ as well as the $\ell^{2}$ norm of the restriction of $E^{N}$ to the interior subinterval $[0,1.2]$ of $(a,b)$. Finally the $L^{2}(a,b)$ norm of the error $e^{N}$ of the piecewise constant approximation, which we compare with the normalized $\ell^{2}$ norm of $E^{N}$, 
$N^{-\frac12}\DNorm{E^{N}}{\ell^{2}(\omega^{N})}$. 
What we observe for the three examples can be compared with the error estimates in Proposition~\ref{P:cons} and Theorems~\ref{T:conv0} and \ref{T:conv}, taking into account that we only have \emph{global} error estimates. Whereas we have seen in Section~\ref{SS:ConsErr} that the consistency error $c^{N}$ behaves differently in the interior of the domain and near the boundary, our technique of proof does not allow to infer the same behavior for the discrete error $E^{N}$. Yet we see for both examples 1 and 2 a $O(N^{-\frac12})$ behavior for $\DNorm{E_{\rm int}^{N}}{\ell^{2}}$, which would correspond to an order of $O(N^{-1})$ for the $\ell^{\infty}$ norm and also for the $L^{2}$ error of the corresponding piecewise constant approximation. 

For the global $\ell^{2}$ norms of the errors $c^{N}$ and $E^{N}$ we see that in example 1 they do not tend to zero, in accordance with the behavior near the boundary of estimates \eqref{E:const} and \eqref{E:consCalpha}. In example 2, they are of order $O(N^{-\frac12})$, which reflects again the boundary layer behavior for the approximation of the function $u_{1}$, which is H\"older continuous of class $C^{\frac12}$ there. 

The global $L^{2}$ norm of the error $e^{N}$ is seen to behave as $O(N^{-\frac12})$ in example 1, which corresponds to Theorem~\ref{T:conv}. In example 2, the behavior is $O(N^{-1})$, which better  than what Theorem~\ref{T:conv} predicts for a $C^{\frac12}$ solution. The explanation is that whereas the Theorems~\ref{T:conv0} and \ref{T:conv}  consider worst-case scenarios, the solution $u_{1}$ is in fact $C^{\infty}$ except at the two boundary points, so that estimate \eqref{E:DErrCov} may be seen to hold with $\alpha=1$.

\begin{figure}[h]
\centering
\includegraphics[width=0.32\textwidth]{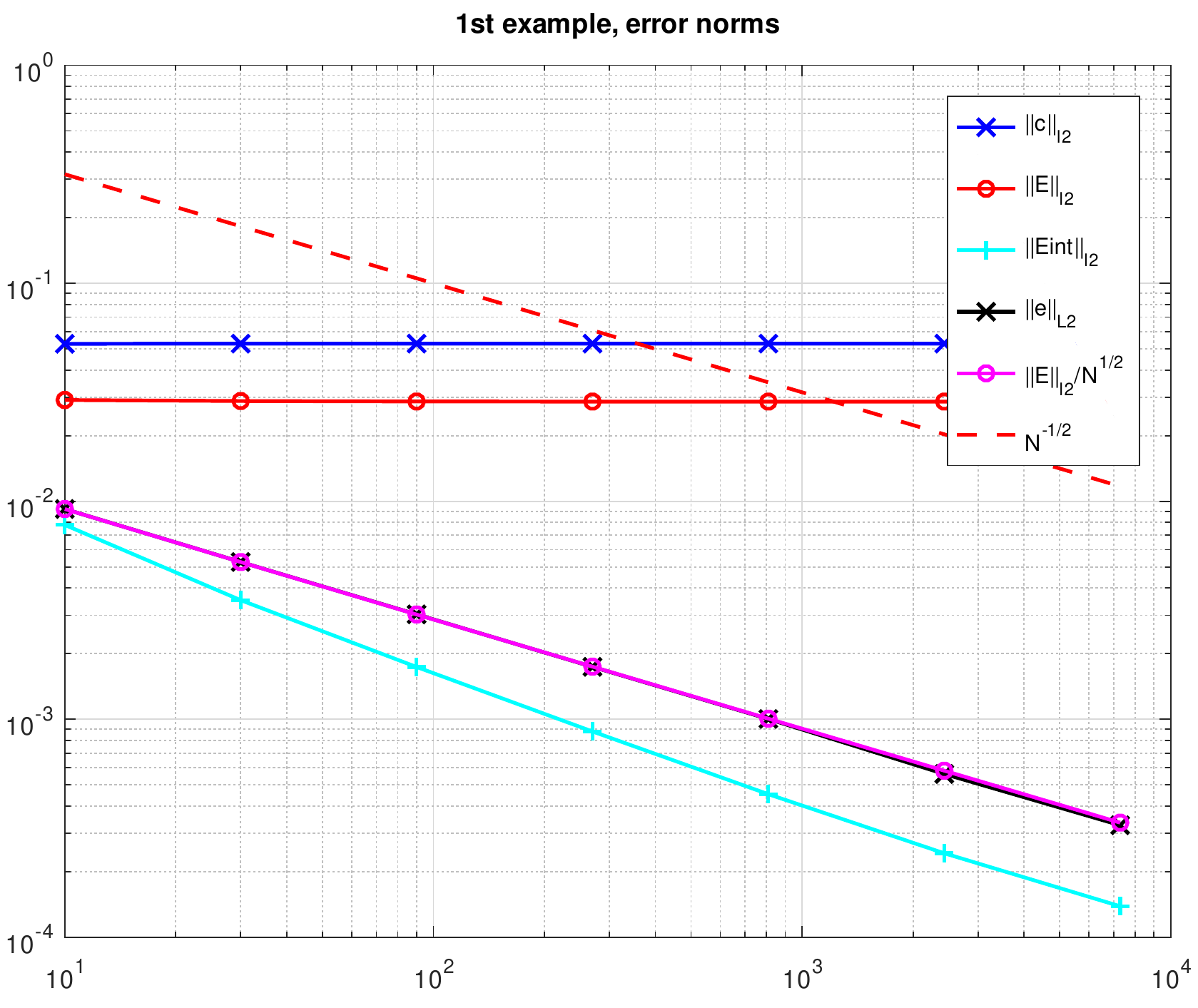}
\includegraphics[width=0.32\textwidth]{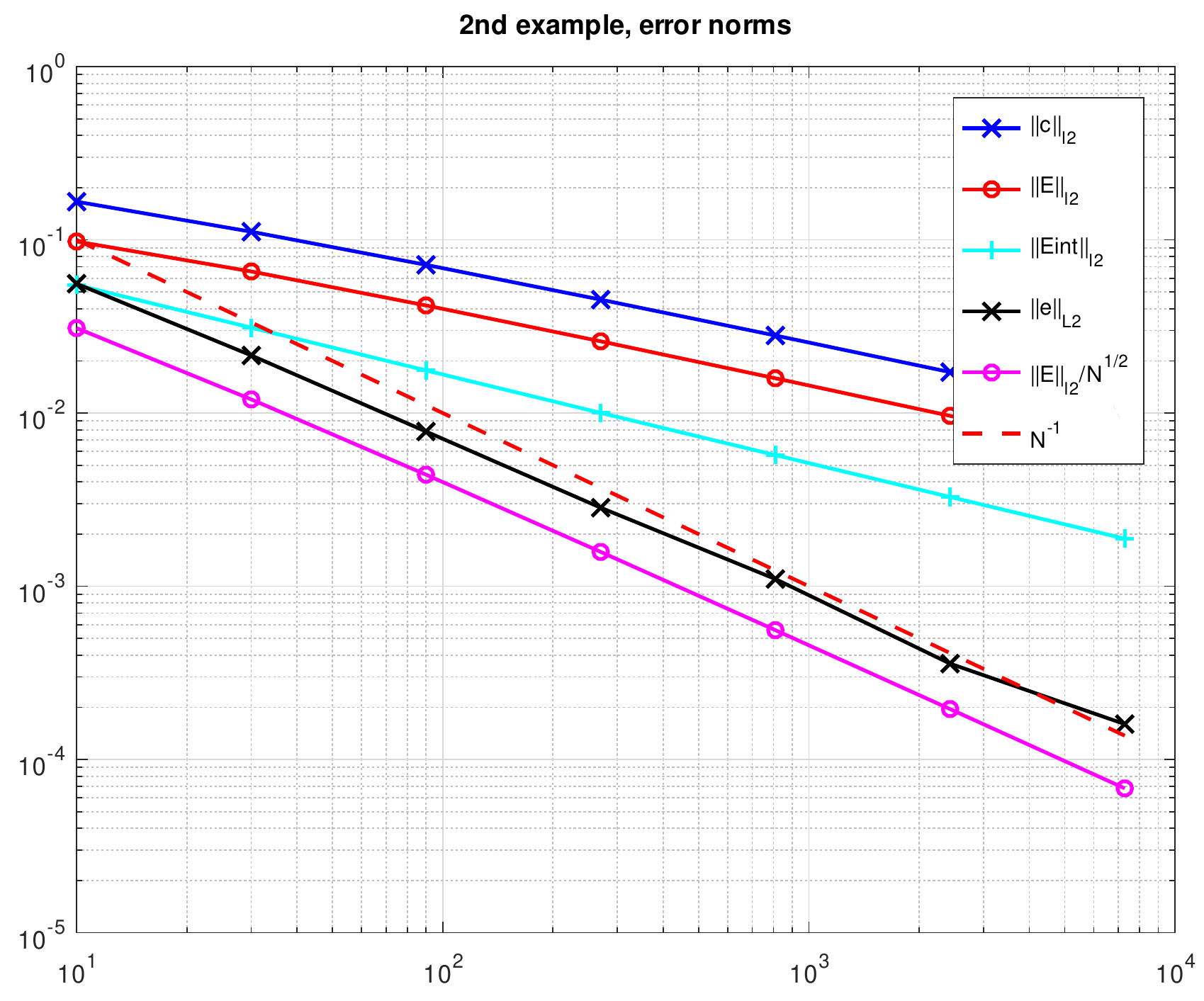}
\includegraphics[width=0.32\textwidth]{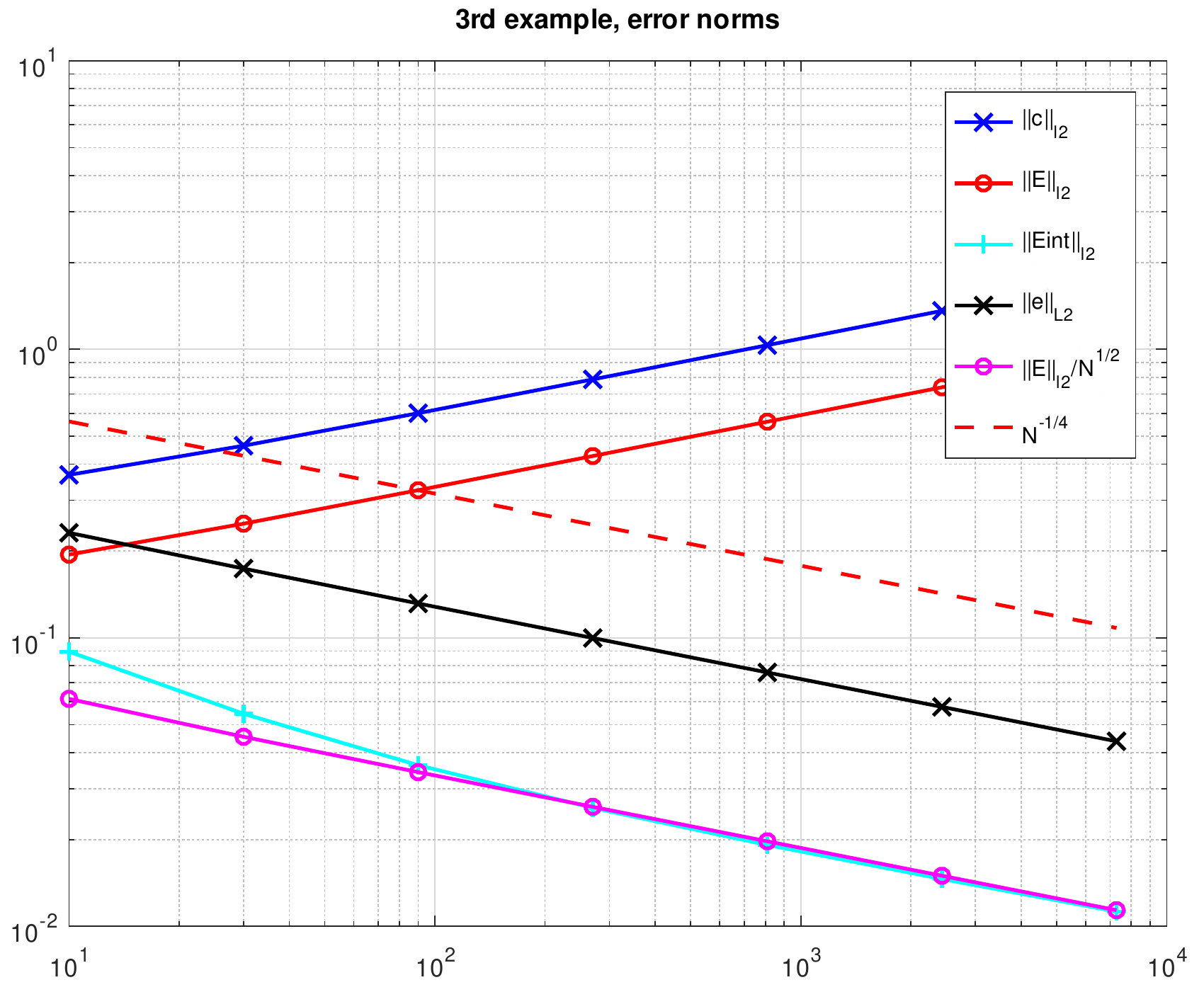}
\caption{Error norms, compared to $N^{-1/2}$ (Ex. 1),  $N^{-1}$ (Ex. 2), $N^{-1/4}$ (Ex. 3),}
\label{F:Ex1-3,errnorms}
\end{figure}

%%%%%%%%%
%\section{}
%%%%%%%%%

%%%%%%%%%
%\subsection{}
%%%%%%%%%

%%%%%%%%%
%\section{}
%%%%%%%%%

%%%%%%%%%
%\subsection{}
%%%%%%%%%

%%%\section{Les r\'ef\'erences}
%%%%\nocite{*}
\bibliographystyle{acm}
\bibliography{database.bib}

\end{document}